\documentclass{amsart}
\usepackage{url}
\usepackage{graphicx}
\newtheorem{thm}{Theorem}[section]
\newtheorem{cor}[thm]{Corollary}
\newtheorem{lem}[thm]{Lemma}

\theoremstyle{definition}

\theoremstyle{remark}

\begin{document}
\title[On some geometric properties of quasi-product production models]{On some geometric properties of quasi-product production models}
\author[H. Alodan, B.-Y. Chen, S. Deshmukh, G.E. V\^{\i}lcu]{Haila Alodan, Bang-Yen Chen, Sharief Deshmukh, Gabriel-Eduard V\^{\i}lcu}

\date{}
\maketitle

\abstract  In this article we obtain classification results on the quasi-product production functions in terms of the geometry of their associated graph hypersurfaces
, generalizing in a new setting some recent results concerning basic production models. 
In particular, we obtain several results on the geometry of Spillman-Mitscherlich and transcendental production functions.
\endabstract

\section{Introduction}

The concept of \emph{production function} is a basic one in economics,  being used in the modeling of the relationship between the output and the inputs of a production process.
From a mathematical point view, a production function is a twice differentiable mapping $f$ from a domain
$D$ of $\mathbb{R}^n_+=\{(x_1,\ldots,x_n)\in\mathbb{R}^n:x_1>0,\ldots,x_n>0\}$ into $\mathbb{R}_+=\{x\in\mathbb{R}:x>0\}$, where $\mathbb{R}$ denotes the set of real numbers. Hence we have $f:D\subset\mathbb{R}^n_+\rightarrow\mathbb{R}_+$, $f=f(x_1,\ldots,x_n)$,
where $f$ is the quantity of output, $n$ is the number of the inputs
and $x_1,\ldots,x_n$ are the factor inputs, such as: labor, capital,
land, raw materials etc. We note that some historical information about the evolution of the concept of production function and a lot of interesting examples can be found in \cite{BHA}. We only recall that, among the classes of production models, the most famous is the Cobb-Douglas (CD) production function introduced in \cite{CD} in order to describe the distribution of the national income of the USA. 
In its generalized form with $n$ inputs, the CD production function is expressed by  \cite{WF}
\begin{equation}\label{2}
f(x_1,...,x_n)=A\cdot
\prod_{i=1}^n x_i^{k_{i}},
\end{equation}
where $A>0$, $k_1,\ldots,k_n\neq0$. It is obvious that the function $f$ given by \eqref{2}, which is usually called the generalized CD production function, is homogeneous of degree $p=\displaystyle\sum_{i=1}^n k_{i}$. We recall that the homogeneity has a precise economic interpretation: 
the multiplication of the inputs by same value $\lambda>0$ leads to a multiplication of the production  by $\lambda^p$, where $p$ denotes the degree of homogeneity. It is known that the production function exhibits \emph{constant return to scale}, shortly CRS, if the degree of homogeneity is $p=1$. Similarly, an \emph{increased return to scale} (\emph{decreased return to scale}) occurs when the degree of homogeneity is $p>1$ ($p<1$).

CD production functions were generalized by H. Uzawa \cite{UZ} and D.
McFadden \cite{MCF} under the form
\begin{equation}\label{999}
f(x_1,...,x_n)=A\left(\sum_{i=1}^{n}k_ix_i^\rho\right)^{\frac{\gamma}{\rho}},\ (x_1,\ldots,x_n)\in D\subset \mathbb{R}^n_+,
\end{equation}
with $A,k_1,\ldots,k_n,\rho\neq0$, where $\gamma$ is the degree of homogeneity. We note that the function $f$ defined by \eqref{999} is called the generalized CES production function.

It is well known that the classical treatment of the production functions makes use of the projections of production functions on a plane, but, unfortunately, this approach leads to limited conclusions and a differential geometric treatment is more than useful. We note that this approach is feasible since any production function $f$ can be identified with the graph of $f$, \emph{i.e.} the nonparametric hypersurface of the $(n+1)$-dimensional
Euclidean space $\mathbb{E}^{n+1}$ defined by
\begin{equation}\label{3}
L(x_1,\ldots,x_n)=(x_1,\ldots,x_n,f(x_1,\ldots,x_n))
\end{equation}
and called the \emph{production hypersurface} of $f$ \cite{VV}. Using this treatment, a surprising link between some basic concepts in the theory of production functions and the differential geometry of hypersurfaces was obtained in \cite{VGE}: a
generalized CD production function has
decreasing/increasing return to scale if and only if the
corresponding hypersurface has positive/negative Gauss-Kronecker curvature.
Moreover, this production function has constant return to scale if
and only if the corresponding hypersurface has vanishing Gauss-Kronecker
curvature. Moreover, in \cite{CV}, the authors proved that a homogeneous production function with
an arbitrary number of inputs defines a flat hypersurface if and only if either it has
constant return to scale or it is a multinomial production function. This result was generalized by X. Wang to the case of homogeneous hypersurfaces with
constant sectional curvature \cite{W}. On the other hand, other classes of production functions, like quasi-sum production functions and homothetic production functions, were investigated via geometric properties of their associated graph hypersurfaces in
Euclidean spaces (see, e.g., \cite{AMih,CH4,CH2} and references therein). We outline that such kind of results are of great interest not only in economic analysis \cite{AOT,INO1}, but also in the classical differential geometry, where the study of hypersurfaces
with certain curvature properties is one of the basic problems \cite{CH7}.

Motivated by the above works, in the present paper we derive the main properties of quasi-product production models in economics
in terms of the geometry of their graph hypersurfaces
, generalizing in a new setting some recent results concerning quasi-sum and homothetic
production models \cite{CH4,CH2,VV3}.

\section{Preliminaries on the geometry of hypersurfaces}


In this section we recall some basic concepts and results concerning the geometry of hypersurfaces in Euclidean spaces, based mainly on \cite{CH7,CH1}.


Let $M$ be a hypersurface of the Euclidean space $\mathbb E^{n+1}$. Then it is known that the \emph{shape operator} $S$ of $M$, which can be defined  using the Gauss map of the hypersurface, is a symmetric endomorphism of the tangent space $T_pM$, for each $p\in M$, playing a key role in the differential geometry of hypersurfaces (see, e.g., the recent papers \cite{DeepA,LiuY,UpadT}). The eigenvalues $\rho_1,\ldots,\rho_n$ of the shape
operator are called \emph{principal curvatures}. The determinant of the shape operator
$S$, denoted by $K$, is called the \emph{Gauss-Kronecker curvature}. When $n = 2$, the
Gauss-Kronecker curvature is simply called the \emph{Gauss curvature}. We recall that a \emph{developable surface} is a surface having null Gauss curvature.  
On another  hand, the \emph{mean curvature} of $M$, denoted by $H$, is defined to be
the average of the principle curvatures, i.e.,
\[
H=\frac{1}{n}\sum_{i=1}^{n}\rho_i.
\]
A \emph{minimal} hypersurface is a hypersurface with vanishing mean curvature.

We recall that the Riemann curvature tensor $R$ of $M$
assigns to three vector fields $(u,v,w)$ on $M$ the vector field
\[
R(u,v)w=\nabla_u\nabla_vw-\nabla_v\nabla_uw-\nabla_{[u,v]}w,
\]
where $\nabla$ is the Levi-Civita connection associated with the induced metric $g$ on the hypersurface from the Euclidean metric on $\mathbb E^{n+1}$.
We say that $M$ is \emph{flat} if its curvature
tensor $R$ is zero at every point.

Next, we denote the partial derivatives $\frac{\partial f}{\partial x_i}$, $\frac{\partial^2 f}{\partial x_i\partial x_j}$,$\ldots$, etc. by $f_{x_i},f_{x_ix_j}$,$\ldots$, etc. We also put
\[
w=\sqrt{1+\displaystyle\sum_{i=1}^{n} f_{x_i}^{2}}.
\]

Next we recall the following well-known result for later use.

\begin{lem}\label{L1} \cite{CH4}
For the production hypersurface of $\mathbb E^{n+1}$ defined by (\ref{3}), one has the following.
\begin{enumerate}
  \item[i.] The Gauss-Kronecker curvature $K$ is given by
      \begin{equation}\label{11}
     K=\frac{\det(f_{x_ix_j})}{w^{n+2}}.
      \end{equation}
  \item[ii.] The mean curvature $H$ is given by
       \begin{equation}\label{11b}
     H=\frac{1}{n}\sum_{i=1}^{n}\frac{\partial}{\partial x_i}\left(\frac{f_{x_i}}{w}\right).
      \end{equation}
  \item[iii.] The sectional curvature $K_{ij}$ of the plane section spanned by $\frac{\partial}{\partial x_i}$, $\frac{\partial}{\partial x_j}$ is
        \begin{equation}\label{12}
      K_{ij}=\frac{f_{x_ix_i}f_{x_jx_j}-f^2_{x_ix_j}}{w^2\left(1+f_{x_i}^2+f_{x_j}^2\right)}.
      \end{equation}
    \item[iv.] The Riemann curvature tensor $R$ and the metric tensor $g$ satisfy
 \begin{equation}\label{13}
 g\left(R\left(\frac{\partial}{\partial x_{i}},\frac{\partial}{\partial x_{j}}\right)\frac{\partial}{\partial x_{k}},\frac{\partial}{\partial x_{\ell}}\right) =\frac{f_{x_ix_{\ell}}f_{x_jx_k}- f_{x_ix_k}f_{x_jx_{\ell}}}{w^{4}}.
\end{equation}
\end{enumerate}
\end{lem}

\section{Quasi-product production models}

Two classes of production functions have been investigated more thoroughly in economics, namely homogeneous and homothetic production functions. Various geometric properties of these production models were obtained in the last period of time by many geometers, but there are some non-homogeneous production functions, including the famous Spillman-Mitscherlich and transcendental production functions, which were not enough investigated from a differential geometric point of view.

We recall that the Spillman-Mitscherlich production function is defined by
\begin{equation}\label{A1}
f(x_1,\ldots,x_n)= A\cdot\left[1-\exp(-a_1x_1)\right]\cdot\ldots\cdot\left[1-\exp(-a_nx_n)\right],
\end{equation}
where $A,a_1,\ldots,a_n$ are positive constants. On the other hand, the transcendental production function is given by
\begin{equation}\label{A2}
f(x_1,\ldots,x_n)= A\cdot x_{1}^{a_1}\exp(b_1x_1)\cdot\ldots\cdot x_{n}^{a_n}\exp(b_nx_n),
\end{equation}
where $A$ is a positive constant and $a_1,b_1,\ldots,a_n,b_n$ are real constants (usually taken from the closed interval $[0,1]$), such that
\[a_i^2+b_i^2\neq 0,\ i=1,\ldots,n.\]

We remark that the Spillman-Mitscherlich and transcendental production functions belong to a more general class of production functions, namely of the form
\begin{equation}\label{A5}
f(x_1,\ldots,x_n)=\prod_{i=1}^{n}g_i(x_i),
\end{equation}
where $g_1,\ldots,g_n$ are continuous positive real functions  with nowhere zero first derivatives. We say that a production function of
the form \eqref{A5} is a \emph{product} production function. In particular, we note that the generalized CD production function
is also a particular type of product production function. Notice that the graph hypersurfaces associated with the product functions having the form \eqref{A5} are of interest not only in economic analysis, but also in classical differential geometry, where they were investigated under the name of factorable hypersurfaces or homothetical hypersurfaces \cite{JIU,LIU,LM,VW}.

Next, we remark that the concept of product production function can be also generalized as follows.
We say that a production model is \emph{quasi-product}, if it is given by
\begin{equation}\label{A6}
f(x_1,\ldots,x_n)=F\left(\prod_{i=1}^{n}g_i(x_i)\right),
\end{equation}
where $F,g_1,\ldots,g_n$ are continuous
positive functions with nowhere zero first derivatives. It is clear that the class of quasi-product production models reduces to the class of product production models, provided that $F$ is the identity function. Moreover, it is to see that the generalized
CES production function is also a particular type of quasi-product production model. We notice that quasi-product productions models which possed flat production hypersurfaces were recently classified by Y. Fu and W.G. Wang in three classes (see \cite[Theorem 3.3]{FW}).

We recall that, if $f$ is a production function with $n$ inputs $x_1,x_2,...,x_n$, $n\geq 2$, then the \emph{output elasticity}, also called elasticity of production  (or output), with respect to an input $x_i$, is given by
\begin{equation}\label{B4}
E_{x_i}=\frac{x_i}{f}f_{x_i},
\end{equation}
while the \emph{marginal rate of technical substitution} between two inputs $x_j$
and $x_i$ is defined as
\begin{equation}\label{B5}
{\rm MRS}_{ij}=\frac{f_{x_j}}{f_{x_i}}.
\end{equation}

A production function is said to satisfy the \emph{proportional marginal rate of substitution} (PMRS) property if and only if
\begin{equation}\label{B6}
{\rm MRS}_{ij}=\frac{x_i}{x_j},\ 1\leq i\neq j\leq n.
\end{equation}

On another hand, it is well known that the most common quantitative indices of production factor substitutability are
forms of the elasticity of substitution \cite{MCF}. We recall that there are two concepts of elasticity of substitution: Hicks elasticity of substitution, denoted by $H_{ij}$ and Allen elasticity of substitution, denoted by $A_{ij}$, $i,j\in\{1,\ldots,n\}$, $i\neq j$. The function $H_{ij}:\mathbb{R}^n_+\rightarrow\mathbb{R}$ defined  by
\begin{equation}\label{B7}
H_{ij}(x_1,\ldots,x_n)=\frac{\frac{1}{x_if_{x_i}}+\frac{1}{x_jf_{x_j}}}{-\frac{f_{x_ix_i}}{f_{x_i}^2}+\frac{2f_{x_ix_j}}{f_{x_i}f_{x_j}}-\frac{f_{x_jx_j}}{f_{x_j}^2}},\
\end{equation}
for all $(x_1,\ldots,x_n)\in\mathbb{R}^n_+$, is known as \emph{Hicks elasticity of substitution}
between the inputs $x_i$
and $x_j$, while
the function $A_{ij}:\mathbb{R}^n_+\rightarrow\mathbb{R}$ given by
\[
A_{ij}(x_1,\ldots,x_n)=-\frac{x_1f_{x_1}+\ldots+x_nf_{x_n}}{x_ix_j}\frac{\Delta_{ij}}{\Delta},
\]
for all $(x_1,\ldots,x_n)\in\mathbb{R}^n_+$, is called the \emph{Allen elasticity of substitution} between the inputs $x_i$
and $x_j$, where $\Delta$ is the determinant of the bordered matrix
\[
\left(
  \begin{array}{cccc}
    0 & f_{x_1} & \ldots & f_{x_n} \\
    f_{x_1} & f_{x_1x_1} & \ldots & f_{x_1x_n} \\
    \vdots & \vdots & \vdots & \vdots \\
    f_{x_n} & f_{x_nx_1} & \ldots & f_{x_nx_n} \\
  \end{array}
\right)
\]
and $\Delta_{ij}$ is the co-factor of the element $f_{ij}$ in the determinant $\Delta$. Notice that the determinant $\Delta$, usually called \emph{Allen determinant}, is assumed to be $\neq 0$. 

It is worth mentioning that $H_{ij}=A_{ij}$, provided that $n=2$. Because of this fact, in case of a production function with only two inputs, both $H_{ij}$ and $A_{ij}$ are called \emph{elasticity of substitution} between the two inputs. On the other hand, for $n\geq3$, it is clear that $H_{ij}\neq A_{ij}$. We note that the elasticity of substitution was originally introduced by J.R. Hicks \cite{HIC} in case of two inputs for the purpose of analyzing
changes in the income shares of labor and capital.

A twice differentiable production function $f$ with nowhere zero first partial derivatives is said to satisfy the  \emph{constant elasticity of substitution} (CES) property if there
is a nonzero real constant $\sigma$ such that
\begin{equation}\label{B9}
H_{ij}(x_1,\ldots,x_n)=\sigma,
\end{equation}
for $(x_1,\ldots,x_n)\in\mathbb{R}^n_+$ and $1\leq i\neq j\leq n$.

We remark that B.-Y. Chen \cite{CH} has completely classified homogeneous production functions which
satisfy the CES property, generalizing to an arbitrary number of inputs an earlier result of L. Losonczi \cite{LL} for two inputs. Moreover, the classification has been later extended to the classes of quasi-sum and homothetic production functions in \cite{CH5,CH6}. We remark that quasi-sum production functions are of special interest, they arise in a natural way in the problem of consistent aggregation \cite{AM}. On the other hand, A. Mihai, M.E. Aydin and M. Erg\"{u}t classified quasi-sum and quasi-product production functions by their Allen determinants \cite{AE2,AMih}.

\section{Some classification results}

\begin{thm}\label{T1}
Suppose $f$ is a quasi-product production model having the form (\ref{A6}), such that $F,g_1,\ldots,g_n$ are twice differentiable functions on their domains of definition. Then the following assertions hold.
\begin{enumerate}
  \item[i.] The output elasticity with respect to an input $x_i$ is a constant $k_i$ iff the production model reduces to
\begin{equation}\label{C7}
f(x_1,\ldots,x_n)=A\cdot x_i^{k_i}\cdot\prod_{j\neq i} g^k_j(x_j),
\end{equation}
where $A$ and $k$ are real constants with $A>0$ and $k\neq0$.
\item[ii.] The output elasticity is constant with respect to all inputs iff $f$ reduces to the generalized CD production function defined by \eqref{2}.
  \item[iii.] The production model satisfies the PMRS property  iff it reduces to
  the following homothetic generalized CD production function:
\begin{equation}\label{C8}
f(x_1,\ldots,x_n)=F\left(A\cdot\prod_{i=1}^n x_i^k\right),
\end{equation}
where $A$ and $k$ are real constants with $A>0$ and $k\neq0$.
\item[iv.] If the production model satisfies the PMRS property, then:
\begin{enumerate}
  \item[iv$_1$.] The production hypersurface associated with the quasi-product production model $f$ is non-minimal.
  \item[iv$_2$.] The production hypersurface associated with the quasi-product production model $f$ has null
sectional curvature iff, up to a suitable translation, $f$ reduces to a generalized CD production function given by
\begin{equation}\label{C100}
    f(x_1,\ldots,x_n)=A\cdot\sqrt{x_1\ldots x_n},
\end{equation}
where $A>0$.
\end{enumerate}
\item[v.] The production hypersurface associated with the quasi-product production model $f$  has null Gauss-Kronecker curvature iff, up to a suitable translation, the function $f$ reduces to the one of the following:
\begin{enumerate}
  \item[(a)] a generalized CD production function with CRS;
  \item[(b)] $f(x_1,\ldots,x_n)=A\cdot\ln\left[\exp{(A_1x_1)}\cdot\displaystyle\prod_{j=2}^ng_j(x_j)\right]$, where $A,A_1$ are nonzero real constants;
  \item[(c)] $f(x_1,\ldots,x_n)=F\left(A\cdot\exp{(A_1x_1+A_2x_2)}\cdot\displaystyle\prod_{j=3}^ng_j(x_j)\right)$, where $A$ is a positive constant and $A_1,A_2$ are nonzero real constants;
  \item[(d)] a generalized CES production function with CRS, given by
        \[f(x_1,\ldots,x_n)=\left(\displaystyle\sum_{i=1}^{n}C_ix_i^{\frac{A}{A-1}}\right)^{\frac{A-1}{A}},\] where $A$ is a nonzero real constant, $A\neq1$, and $C_1,\ldots,C_n$ are nonzero real constants;
  \item[(e)] $f(x_1,\ldots,x_n)=A\cdot\ln\left(\displaystyle\sum_{i=1}^{n}B_i\exp(A_ix_i)\right)$,  where $A,A_i,B_i$ are nonzero real constants for $i=1,\ldots,n$.
\end{enumerate}
\item[vi.] The production model satisfies the CES property iff, up to a suitable translation, the function $f$ reduces to the one of the following:
\begin{enumerate}
    \item[(a)] a homothetic generalized CD production function given by
    \[
    f(x_1,\ldots,x_n)=F\left(A\cdot\prod_{i=1}^n x_i^{k_i}\right),
    \]
where $A$ is a positive constant and $k_1,\ldots,k_n$ are nonzero real constants.
    \item[(b)] $f(x_1,\ldots,x_n)=F\left(A\cdot\displaystyle\prod_{i=1}^n\exp\left(A_ix_i^{\frac{\sigma-1}{\sigma}}\right)\right)$,  where $A$ is a positive constant and $A_1,\ldots,A_n,\sigma$ are nonzero real constants, $\sigma\neq 1$;
    \item[(c)] a two-input production function given by \[f(x_1,x_2)=F\left(A\cdot\left(\frac{x_1^{\frac{\sigma-1}{\sigma}}+A_1}{x_2^{\frac{\sigma-1}{\sigma}}+A_2}\right)^{\frac{\sigma}{k}}\right),\] where $A,A_1,A_2,k,\sigma$ are nonzero real constants, $\sigma\neq 1$;
    \item[(d)] a two-input production function given by \[f(x_1,x_2)=F\left(A\cdot\left(\frac{\ln (A_1x_1)}{\ln (A_2x_2)}\right)^{\frac{1}{k}}\right),\] where $A,k$ are nonzero real constants and $A_1,A_2$ are positive constants.
\end{enumerate}
\end{enumerate}
\end{thm}
\begin{proof}
In what follows we will use the notation $u=g_1(x_1)\cdot\ldots\cdot g_n(x_n)$. Then we have
 \begin{equation}\label{E14}
f_{x_i}=uF' \frac{g'_i}{g_i},
\end{equation}
where $F'$ denotes the derivative with respect to the variable $u$ and $g'_i=\frac{{\rm d}g_i}{{\rm d}x_i}$, for $i=1,\ldots,n$.

From (\ref{E14}) we derive
 \begin{equation}\label{E15}
f_{x_ix_i}=u^2F''\left(\frac{g'_i}{g_i}\right)^2+uF'\frac{g''_i}{g_i},\ i=1,\ldots,n
\end{equation}
and
 \begin{equation}\label{E16}
f_{x_ix_j}=u(uF''+F')\frac{g'_ig'_j}{g_ig_j},\ i\neq j.
\end{equation}

i. We first suppose that the output elasticity is a constant $k_i$ with respect to an input $x_i$. Then we derive from (\ref{B4}) that
 \begin{equation}\label{E17}
f_{x_i}=k_i\frac{f}{x_i}.
\end{equation}

By replacing (\ref{A6}) and (\ref{E14}) in (\ref{E17}) we obtain
 \begin{equation}\label{E18}
u\frac{F'}{F}=\frac{k_i}{x_i}\cdot\frac{g_i}{g'_i}.
\end{equation}

The partial derivative of the expression (\ref{E18}) with respect to $x_j$, $j\neq i$, leads to
\[
u\frac{g'_j}{g_j}\frac{(F'+uF'')F-uF'^2}{F^2}=0.
\]
Hence, because $u>0$ and $g'_j\neq 0$, we deduce that
\begin{equation}\label{E19}
\left(\frac{uF'}{F}\right)'=0.
\end{equation}

We obtain now easily that the solution of (\ref{E19}) is
\begin{equation}\label{E20}
F(u)=C\cdot u^{k},
\end{equation}
for some constants $C>0$ and $k\in\mathbb{R}-\{0\}$.
From (\ref{E20}) and (\ref{E18}) we derive
\[
\frac{g'_i}{g_i}=\frac{k_i}{kx_i}
\]
with solution
\begin{equation}\label{E21}
g_i(x_i)=B\cdot x_i^{\frac{k_i}{k}},
\end{equation}
where $B$ is a positive constant.

Finally, combining (\ref{A6}), (\ref{E20}) and (\ref{E21}) we get a function of the form (\ref{C7}), where $A=C\cdot B^k$.

The right-to-left implication follows immediately by straightforward computation.

ii. The proof is clear from i.

iii. Let us assume that $f$ satisfies the PMRS property. Then taking account of (\ref{B5}), (\ref{B6}) and (\ref{E14}) we obtain
\[
x_j\frac{g'_j}{g_j}=x_i\frac{g'_i}{g_i},\ \forall i\neq j.
\]

Therefore we derive that there exists a real constant $k\neq0$ such that
\[
x_i\frac{g'_i}{g_i}=k,\ i=1,\ldots,n,
\]
and we get
\begin{equation}\label{E22}
g_i(x_i)=A_ix_i^k,\ i=1,\ldots,n,
\end{equation}
for some positive constants $A_1,\ldots,A_n$.

From (\ref{A6}) and (\ref{E22}) we derive that
\[
f(x_1,\ldots,x_n)=F\left(A\displaystyle\prod_{i=1}^{n}x_i^k\right),
\]
where \[A=\displaystyle\prod_{i=1}^{n}A_i\] and the conclusion follows.

The right-to-left implication can be easily checked by direct computation.

iv. We assume that the production function given by \eqref{A6} satisfies the PMRS property. Then we deduce
from \eqref{C8} that, denoting $G=F\circ g$, where  $g(t)=A \cdot t^k$, the function $f$ takes the form
\begin{equation}\label{x33}
f(x_1,\ldots,x_n)=G\left(\displaystyle\prod_{i=1}^nx_i\right).
\end{equation}

Therefore we have
\begin{equation}\label{37}
f_{x_i}=\frac{uG'}{x_i},
\end{equation}
\begin{equation}\label{38}
f_{x_ix_i}=\frac{u^2G''}{x_i^2}
\end{equation}
and
\begin{equation}\label{39}
f_{x_ix_j}=\frac{u(G'+uG'')}{x_ix_j},
\end{equation}
where \[u=\displaystyle\prod_{i=1}^nx_i.\]

If the corresponding production hypersurface of $f$ is minimal, then $H=0$ and from (\ref{11b}) we obtain
\begin{equation}\label{36}
\sum_{i=1}^{n}f_{x_ix_i}+\sum_{i\neq j}\left(f_{x_i}^2f_{x_jx_j}-f_{x_i}f_{x_j}f_{x_ix_j}\right)=0.
\end{equation}

Making use of \eqref{37}, \eqref{38} and \eqref{39} in \eqref{36} we get
\begin{equation}\label{40}
u^2G''\sum_{i=1}^{n}\frac{1}{x_i^2}-u^3G'^3\sum_{i\neq j}\frac{1}{x_i^2x_j^2}=0.
\end{equation}

We can see now that the unique solution of the equation (\ref{40}) is $G(u)=constant$, which is clearly a contradiction. Hence the assertion (iv$_1$) follows.

Let us assume now that the  production hypersurface of $f$ has $K_{ij}=0$. Then we derive from (\ref{12}) that
\begin{equation}\label{x32}
f_{x_ix_i}f_{x_jx_j}-f^2_{x_ix_j}=0.
\end{equation}

Making use of \eqref{37}, \eqref{38} and \eqref{39} in (\ref{x32}),
and taking into account that $G'\neq 0$, we obtain
\[
\frac{G''}{G'}=-\frac{1}{2u}.
\]
Therefore we get immediately
\begin{equation}\label{x35}
G(u)=A\sqrt{u}+B
\end{equation}
for some constants $A,B$, with $A\neq0$.

Combining now (\ref{x33}) and (\ref{x35}), we obtain that, modulo a translation, $f$ reduces to a generalized
CD production function having the form (\ref{C100}).
Since the converse can be easily verified by direct computation, the assertion (iv$_2$) follows.

v. We first suppose that the graph hypersurface associated with the quasi-product production model $f$ has vanishing Gauss-Kronecker curvature. Then we derive from (\ref{11}) that
\begin{equation}\label{E23a}
\det(f_{x_ix_j})=0.
\end{equation}

But using \eqref{E15} and \eqref{E16}, we obtain that the Hessian matrix of a composite
function of the form \eqref{A6} has the determinant expressed by \cite{AE2014}
\begin{equation}\label{E23b}
\det(f_{x_ix_j})=(uF')^n\left[\prod_{j=1}^n \left(\frac{g_j'}{g_j}\right)'+\left(1+u\frac{F''}{F'}\right)\sum_{j=1}^n\left(\left(\frac{g_j'}{g_j}\right)^2\cdot\prod_{i\neq j}\left(\frac{g_i'}{g_i}\right)'\right)\right].
\end{equation}

We divide now the proof of the theorem into two main cases: (A) and (B).\\
\emph{Case} (A): $\frac{g_1'}{g_1},\ldots,\frac{g_n'}{g_n}$ are nonconstant.
Then, from (\ref{E23a}) and (\ref{E23b}) we derive
\begin{equation}\label{E23bb}
1+\left(1+u\frac{F''}{F'}\right)\sum_{j=1}^n\frac{\left(\frac{g_j'}{g_j}\right)^2}{\left(\frac{g_j'}{g_j}\right)'}=0.
\end{equation}

We remark that for the above equation to have solution, it is necessary to have \[1+u\frac{F''}{F'}\neq 0\] and \[\displaystyle\sum_{j=1}^n\frac{\left(\frac{g_j'}{g_j}\right)^2}{\left(\frac{g_j'}{g_j}\right)'}\neq0.\] In this case, \eqref{E23bb} reduces to
\begin{equation}\label{VGE1}
\sum_{j=1}^n\frac{\left(\frac{g_j'}{g_j}\right)^2}{\left(\frac{g_j'}{g_j}\right)'}=-\frac{F'}{F'+uF''}.
\end{equation}

By taking the partial derivative of \eqref{VGE1} with respect to $x_i$ and dividing both sides of the derived expression by $\frac{g_i'}{g_i}$, we obtain
\begin{equation}\label{VGE2}
2-\frac{\frac{g_i'}{g_i}\cdot\left(\frac{g_i'}{g_i}\right)''}{\left[\left(\frac{g_i'}{g_i}\right)'\right]^2}=u\cdot \frac{F'F''+u[F'F'''-(F'')^2]}{(F'+uF'')^2}.
\end{equation}

Therefore, after taking the partial derivative of \eqref{VGE2} with respect to $x_j$, with $j\neq i$, and simplifying the derived expression by  $\left(u\cdot\frac{g_j'}{g_j}\right)$ we get
\begin{eqnarray}\label{VGE4}
\frac{F'F''+u[F'F'''-(F'')^2]}{(F'+uF'')^2}&+&u\cdot\frac{2F'F'''+u(F'F^{iv}-F''F''')}{(F'+uF'')^2}\nonumber\\
&-&2u\cdot\frac{[F'F''+u(F'F'''-(F'')^2)](2F''+uF''')}{(F'+uF'')^3}=0.\nonumber
\end{eqnarray}

We remark now that making the substitution
\[
G=\frac{F'F''+u[F'F'''-(F'')^2]}{(F'+uF'')^2},
\]
the above equation reduces to
\[
G+uG'=0,
\]
with solution \[G(u)=\frac{A}{u},\] where $A$ is a real constant. Hence we derive that
\[
\frac{F'F''+u[F'F'''-(F'')^2]}{(F'+uF'')^2}=\frac{A}{u},
\]
which is equivalent to
\begin{equation}\label{VGE5}
\left(-\frac{F'}{F'+uF''}\right)'=\frac{A}{u}.
\end{equation}

From \eqref{VGE5} we find
\begin{equation}\label{VGE6}
-\frac{F'}{F'+uF''}=A\ln u+B,
\end{equation}
for some real constants $A,B$.

We divide now the proof of case (A) into several cases.\\
\emph{Case} (A.1) $A=0$; In this case it follows that $B\neq0$ and \eqref{VGE6} implies that
\begin{equation}\label{E23c}
1+u\frac{F''}{F'}=-\frac{1}{B}.
\end{equation}

On the other hand, from \eqref{VGE1} we deduce that
\begin{equation}\label{E23cc}
\sum_{i=1}^n\frac{\left(\frac{g_i'}{g_i}\right)^2}{\left(\frac{g_i'}{g_i}\right)'}=B
\end{equation}
for any nonzero real constant $B$.

Solving (\ref{E23cc}) we find
\begin{equation}\label{E23d}
g_i(x_i)=A_i(x_i+B_i)^{-k_i},
\end{equation}
for some constants $A_i,B_i,k_i$, with $A_i\neq 0$ and $k_i\neq 0$, $i=1,\ldots,n$, such that $\displaystyle\sum_{i=1}^{n}k_i=B$.

On the other hand, (\ref{E23c}) reduces to
\begin{equation}\label{E23ccc}
\frac{F''}{F'}=-\frac{B+1}{Bu}.
\end{equation}
\emph{Case} (A.1.1): $B=-1$. Then  from \eqref{E23ccc} we derive that \[F(u)=Cu+D,\] for some real constants $C,D$, with $C\neq0$, and combining with (\ref{A6}) and (\ref{E23d}), we conclude that, after a suitable translation, the function $f$ reduces to a
generalized Cobb-Douglas production function with constant return to
scale. Hence we obtain the case (a) of the theorem.\\
\emph{Case} (A.1.2): $B\neq-1$. Then we obtain easily that the solution of \eqref{E23ccc} is
\begin{equation}\label{E23dd}
F(u)=C \cdot{} u^{-\frac{1}{B}}+D,
\end{equation}
for some real constants $C,D$, with $C\neq 0$.

Combining now (\ref{A6}), (\ref{E23d}) and (\ref{E23dd}), after a suitable translation, we conclude that $f$ reduces to the
following function
\[
f(x_1,\ldots,x_n)=A\cdot\prod_{i=1}^{n}x_i^{\frac{k_i}{B}},
\]
where $A$ is a positive constant.  But it is obvious that \[\displaystyle\sum_{i=1}^{n}\frac{k_i}{B}=1\] and therefore we deduce that the above function is a generalized Cobb-Douglas production function with constant return to
scale. Hence we obtain again the case (a) of the theorem.\\
\emph{Case} (A.2) $A\neq0$; In this case it follows that it is necessary to have \[A\ln u+B\neq0\] and we derive from \eqref{VGE6} that
\[
\frac{F''}{F'}=-\frac{1}{u(A\ln u+B)}-\frac{1}{u}.
\]
Hence we obtain
\begin{equation}\label{ddd}
F'(u)=\frac{C}{u(A\ln u+B)^{\frac{1}{A}}},
\end{equation}
where $C$ is a nonzero real constant. Now, from \eqref{ddd}, we get that
\begin{equation}\label{dddd}
F(u)=D(\ln u+E)^{-\frac{1}{A}+1}+F,
\end{equation}
where $D$ is a nonzero real constant and $E,F$ are real constant, provided that $A\neq1$. On the other hand, if $A=1$, then we obtain from  \eqref{ddd} that
\begin{equation}\label{dddd2}
F(u)=C\ln(\ln u+B)+D,
\end{equation}
where $D$ is a real constant.

But we can easily see now that \eqref{VGE2} implies
\begin{equation}\label{VGE200}
2-\frac{\frac{g_i'}{g_i}\cdot\left(\frac{g_i'}{g_i}\right)''}{\left[\left(\frac{g_i'}{g_i}\right)'\right]^2}=A,
\end{equation}
for $i=1,\ldots,n$.\\
\emph{Case} (A.2.1) A=2; In this case we obtain from \eqref{dddd}
\begin{equation}\label{ddddd}
F(u)=D\sqrt{\ln u+E}+F,
\end{equation}
and from
\eqref{VGE200} it follows that
\[
\left(\frac{g_i'}{g_i}\right)''=0.
\]
Hence we derive
\begin{equation}\label{VGE20}
g_i(x_i)=\exp(a_ix_i^2+b_ix_i+c_i),\ i=1,\ldots,n,
\end{equation}
where $a_i,b_i,c_i$ are real constants. Because $g_i'\neq 0$ on  $\mathbb{R}_+$, it follows that the constants  $a_i$ and $b_i$ must satisfy the following conditions: $a_ib_i\geq0$ and $a_i^2+b_i^2\neq 0$, for $i=1,\ldots,n$. Combining now \eqref{A6}, \eqref{ddddd} and \eqref{VGE20} we deduce that $f$ takes the form
\begin{equation}\label{VGE21}
f(x_1,\ldots,x_n)=D\cdot\sqrt{\sum_{i=1}^{n}A_i(x_i+B_i)^2+E}+F,
\end{equation}
for some constants $A_i,B_i,E,F$ with $A_i\neq 0$. Now, making use of Lemma \ref{L1}(i), it is direct to verify that the production hypersurface associated with the production function given by \eqref{VGE21} has vanishing Gauss-Kronecker curvature if and only if $E=0$. Hence, after a suitable translation we obtain the case (d) of the theorem with $A=2$.\\
\emph{Case} (A.2.2) $A\neq2$; We deduce from \eqref{VGE200} that
\begin{equation}\label{VGE22}
\frac{\frac{g_i'}{g_i}\cdot\left(\frac{g_i'}{g_i}\right)''}{\left[\left(\frac{g_i'}{g_i}\right)'\right]^2}=2-A.
\end{equation}

Denoting $h_i=\frac{g_i'}{g_i}$, we obtain from \eqref{VGE22}
\begin{equation}\label{VGE23}
\left(\frac{h_i}{h'_i}\right)'=A-1
\end{equation}
\emph{Case} (A.2.2.i.) $A=1$; Then from \eqref{VGE23} we conclude that \[\frac{h_i}{h'_i}=\bar{A}_i,\] where $\bar{A}_i$ is a nonzero real constant ($i=1,\ldots,n$) and we deduce
\[
\frac{g_i'}{g_i}=D_i\exp(A_ix_i),\ i=1,\ldots,n,
\]
where $D_i$ is a real constant and $A_i=(\bar{A}_i)^{-1}$ for $i=1,\ldots,n$. Now we can derive immediately that
\begin{equation}\label{VGE24}
g_i(x_i)=C_i\exp\left(B_i\exp(A_ix_i)\right),\ i=1,\ldots,n,
\end{equation}
where $B_i$ is a nonzero constant and $C_i$ is a positive constant.

Combining now \eqref{A6}, \eqref{dddd2} and \eqref{VGE24} we deduce that $f$ takes the form
\begin{equation}\label{VGE25}
f(x_1,\ldots,x_n)=C\ln\left(\displaystyle\sum_{i=1}^{n}B_i\exp(A_ix_i)+B\right)+D
\end{equation}
for some nonzero constants $C,A_i,B_i$, $ i=1,\ldots,n$, and real constants $B,D$. Now, making use of Lemma \ref{L1}(i), it follows by direct computation that the production hypersurface associated with the production function given by \eqref{VGE25} has vanishing Gauss-Kronecker curvature if and only if $B=0$. Hence, after a suitable translation, we obtain the case (e) of the theorem.\\
\emph{Case} (A.2.2.ii.) $A\neq1$; Then from \eqref{VGE23} we derive that \[\frac{h_i'}{h_i}=\frac{1}{(A-1)x_i+A_i},\] where $A_i$ is a real constant ($i=1,\ldots,n$) and we obtain
\begin{equation}\label{VGE26}
\frac{g_i'}{g_i}=B_i[(A-1)x_i+A_i]^{\frac{1}{A-1}},\ i=1,\ldots,n,
\end{equation}
where $B_i$ is a nonzero real constant, $i=1,\ldots,n$.

From \eqref{VGE26} we obtain
\begin{equation}\label{VGE30}
g_i(x_i)=C_i\exp\left(\frac{B_i}{A}[(A-1)x_i+A_i]^{\frac{A}{A-1}}\right),\ i=1,\ldots,n,
\end{equation}
where $C_i$ is a positive constant, $i=1,\ldots,n$.

Combining \eqref{A6}, \eqref{dddd} and \eqref{VGE30} we deduce that $f$ takes the form
\begin{equation}\label{VGE31}
f(x_1,\ldots,x_n)=D\left(\sum_{i=1}^{n}\frac{B_i}{A}[(A-1)x_i+A_i]^{\frac{A}{A-1}}+B\right)^{\frac{A-1}{A}}+F
\end{equation}
for some constants $D,A_i,B_i,B,F$, with $D>0$ and $B_i\neq0$. Now, using Lemma \ref{L1}(i), we can easily verify that the production hypersurface associated with the production function given by \eqref{VGE31} has vanishing Gauss-Kronecker curvature if and only if $B=0$. Hence, after a suitable translation we obtain the case (d) of the theorem.\\
\emph{Case} (B): at least one of $\frac{g_1'}{g_1},\ldots,\frac{g_n'}{g_n}$ is constant. Without loss of generality, we may assume that $\frac{g_1'}{g_1}=A_1$, where $A_1$ is a nonzero real constant. Then we derive that
\begin{equation}\label{VGE32}
g_1(x_1)=B_1\exp(A_1x_1),
\end{equation}
where $B_1$ is a positive constant. Then  (\ref{E23a}) and (\ref{E23b}) imply
\begin{equation}\label{VGE33}
\left(1+u\frac{F''}{F'}\right)\cdot\prod_{i=2}^{n}\left(\frac{g'_i}{g_i}\right)'=0.
\end{equation}

From \eqref{VGE33} we derive that either \[1+u\frac{F''}{F'}=0\] or \[\prod_{i=2}^{n}\left(\frac{g'_i}{g_i}\right)'=0.\] But in the first case we get immediately that
\begin{equation}\label{VGE34}
F(u)=A\ln u+B,
\end{equation}
where $A,B$ are real constants, $A\neq0$. Hence, from \eqref{A6}, \eqref{VGE32} and \eqref{VGE34} we deduce that, after a suitable translation, we obtain the case (b) of the theorem.

On the other hand, in the second case we may assume without loss of generality that $\left(\frac{g'_2}{g_2}\right)'=0$. Hence we get
\begin{equation}\label{VGE35}
g_2(x_2)=B_2\exp(A_2x_2),
\end{equation}
where  $A_2$ is a nonzero real constant and $B_2$ is a positive constant.

Combining now \eqref{A6}, \eqref{VGE32} and \eqref{VGE35},  we obtain the case (c) of the theorem.

Conversely, we can verify by direct computation that all of the production hypersurfaces defined by the production functions in cases (a)-(e) of the theorem have vanishing Gauss-Kronecker curvature.

vi. We first assume that the production function satisfies the constant elasticity of substitution property. Then using
\eqref{B7}, \eqref{E14}, \eqref{E15} and \eqref{E16}  in  \eqref{B9} we obtain
\begin{eqnarray}\label{41}
&&\sigma x_ix_ju^3(F')^3\left[2\left(\frac{g'_i}{g_i}\right)^2\left(\frac{g'_j}{g_j}\right)^2-\frac{g''_i}{g_i}\left(\frac{g'_j}{g_j}\right)^2-
\frac{g''_j}{g_j}\left(\frac{g'_i}{g_i}\right)^2\right]=\nonumber\\
&&=u^3(F')^3\frac{g'_i}{g_i}\frac{g'_j}{g_j}\left(x_i\frac{g'_i}{g_i}+x_j\frac{g'_j}{g_j}\right)
\end{eqnarray}
and taking into account that $x_i>0$, $i=1,\ldots,n$, and $F,g_1,\ldots,g_n$ are
positive functions with nowhere zero first derivatives, we get from \eqref{41}:
\begin{equation}\label{42}
\sigma\left[2-\frac{g_ig_i''}{(g_i')^2}-\frac{g_jg_j''}{(g_j')^2}\right]=\frac{1}{x_i}\frac{g_i}{g'_i}+\frac{1}{x_j}\frac{g_j}{g'_j}.
\end{equation}

Now, it is easy to see that \eqref{42} can be written as
\begin{equation}\label{43}
\frac{1}{x_i}\frac{g_i}{g'_i}-\sigma\left(\frac{g_i}{g_i'}\right)'+\frac{1}{x_j}\frac{g_j}{g'_j}-\sigma\left(\frac{g_j}{g_j'}\right)'=0,
\end{equation}
for $1\leq i\neq j\leq n$.

Next, we can divide the proof into two separate cases.\\
\emph{Case} A. $n\geq3$. Then it is obvious that \eqref{43} implies
\begin{equation}\label{44}
\frac{1}{x_i}\frac{g_i}{g'_i}-\sigma\left(\frac{g_i}{g_i'}\right)'=0,\ i=1,\ldots,n,
\end{equation}
and we derive easily that the solution of \eqref{44} is
\begin{equation}\label{45}
g_i(x_i)=\left\{
           \begin{array}{ll}
             B_i\exp\left(C_ix_i^{\frac{\sigma-1}{\sigma}}\right), & \hbox{$\sigma\neq1$} \\
             B_ix_i^{C_i}, & \hbox{$\sigma=1$}
           \end{array}
         \right.,
\end{equation}
for some positive constants $B_i$ and nonzero real constants $C_i$, $i=1,\ldots,n$. Combining now \eqref{A6} and \eqref{45} we get cases (a) and (b) of the theorem.\\
\emph{Case} B. $n=2$. Then it follows from \eqref{43} that
\begin{equation}\label{46}
\left\{
           \begin{array}{ll}
            \frac{1}{x_1}\frac{g_1}{g'_1}-\sigma\left(\frac{g_1}{g_1'}\right)'=k \\
            \frac{1}{x_2}\frac{g_2}{g'_2}-\sigma\left(\frac{g_2}{g_2'}\right)'=-k
           \end{array}
         \right.,
\end{equation}
for some constant $k$. We remark now that, if $k=0$, then we obtain immediately the cases (a) and (b) of the theorem with $n=2$. Next we consider that $k\neq0$. Then solving \eqref{46}, we derive
\begin{equation}\label{47}
g_1(x_1)=\left\{
           \begin{array}{ll}
             B_1\left(\frac{k}{\sigma-1}x_1^{\frac{\sigma-1}{\sigma}}+C_1\right)^{\frac{\sigma}{k}}, & \hbox{$\sigma\neq1$} \\
             B_1(k\ln x_1+C_1)^{\frac{1}{k}}, & \hbox{$\sigma=1$}
           \end{array}
         \right.
\end{equation}
and
\begin{equation}\label{48}
g_2(x_2)=\left\{
           \begin{array}{ll}
             B_2\left(\frac{k}{\sigma-1}x_2^{\frac{\sigma-1}{\sigma}}+C_2\right)^{-\frac{\sigma}{k}}, & \hbox{$\sigma\neq1$} \\
             B_2(k\ln x_2+C_2)^{-\frac{1}{k}}, & \hbox{$\sigma=1$}
           \end{array}
         \right.
\end{equation}
for some constants $B_1,B_2,C_1,C_2$. Combining now \eqref{A6}, \eqref{47} and \eqref{48} we get cases (c) and (d) of the theorem.

The converse follows easily by direct computation.
\end{proof}

From \cite[Theorem 3.3]{FW} and Theorem \ref{T1} we obtain easily the following result for product production models.

\begin{cor}\label{COR2}
Suppose $f$ is a product production model having the form (\ref{A5}), such that $g_1,\ldots,g_n$ are twice differentiable functions on their domains of definition. Then:
\begin{enumerate}
  \item[i.] The output elasticity with respect to an input $x_i$ is a constant $k_i$  iff the production model reduces to
\[
f(x_1,\ldots,x_n)=A\cdot x_i^{k_i}\cdot\prod_{j\neq i} g_j(x_j),
\]
where $A$ and $k$ are real constants with $A>0$ and $k\neq0$.
\item[ii.] The output elasticity is constant with respect to all inputs iff $f$ reduces to a generalized CD production function defined by \eqref{2}.
  \item[iii.] The production model satisfies the PMRS property  iff it reduces to
  a generalized CD production function defined by
\[
f(x_1,\ldots,x_n)=A\cdot\prod_{i=1}^n x_i^k,
\]
where $A$ and $k$ are real constants with $A>0$ and $k\neq0$.
\item[iv.] If the product production model satisfies the PMRS property, then:
\begin{enumerate}
  \item[iv$_1$.] The production hypersurface associated with the product production model is non-minimal.
  \item[iv$_2$.] The production hypersurface associated with the product production model $f$ has null
sectional curvature if and only if, up to a suitable translation, $f$ reduces to a generalized CD production function defined by \eqref{C100}.
\end{enumerate}
\item[v.] The production hypersurface associated with the product production model $f$ has null Gauss-Kronecker curvature if and only if, up to a suitable translation, $f$ reduces to the one of the following:
\begin{enumerate}
  \item[(a)] a generalized CD production function with CRS;
  \item[(b)] $f(x_1,\ldots,x_n)=A\cdot\exp{(A_1x_1+A_2x_2)}\cdot\displaystyle\prod_{j=3}^ng_j(x_j)$, where $A$ is a positive constant and $A_1,A_2$ are nonzero real constants.
\end{enumerate}
\item[vi.] The production model satisfies the CES property iff, up to a suitable translation, $f$ reduces to the one of the following:
\begin{enumerate}
    \item[(a)] a generalized CD production function given by \eqref{2}.
    \item[(b)] $f(x_1,\ldots,x_n)=A\cdot\displaystyle\prod_{i=1}^n\exp\left(A_ix_i^{\frac{\sigma-1}{\sigma}}\right)$, $\sigma\neq 1$,
    where $A$ is a positive constant and $A_1,\ldots,A_n,\sigma$ are nonzero real constants, $\sigma\neq 1$;
    \item[(c)] a two-input production function given by \[f(x_1,x_2)=A\cdot\left(\frac{x_1^{\frac{\sigma-1}{\sigma}}+A_1}{x_2^{\frac{\sigma-1}{\sigma}}+A_2}\right)^{\frac{\sigma}{k}},\]
    where $A,k,\sigma,A_1,A_2$ are nonzero real constants, $\sigma\neq 1$;
    \item[(d)] a two-input production function given by \[f(x_1,x_2)=A\cdot\left(\frac{\ln (A_1x_1)}{\ln (A_2x_2)}\right)^{\frac{1}{k}},\] where $A,k$ are nonzero real constants and $A_1,A_2$ are positive constants.
\end{enumerate}
\item[viii.] The production hypersurface associated with the product production model $f$  is flat iff, up to a suitable translation, $f$ reduces to the one of the following:
\begin{enumerate}
  \item[(a)] $f(x_1,\ldots,x_n)=A\cdot\displaystyle\prod_{i=1}^n\exp\left(C_ix_i\right)$, where $A$ is a positive constant and $C_1,\ldots,C_n$ are nonzero real constants;
  \item[(b)] A generalized Cobb-Douglas production function given by \[f(x_1,\ldots,x_n)=A\sqrt{x_1\cdot\ldots\cdot x_n},\] where $A$ is a positive constant.
\end{enumerate}
\end{enumerate}
\end{cor}

From Corollary \ref{COR2} we obtain the following results concerning Spillman-Mitscherlich and transcendental
production functions.

\begin{cor}\label{C2} Suppose $f$ is a Spillman-Mitscherlich  production model given by \eqref{A1}. Then the following assertions hold.
\begin{enumerate}
  \item[i.] The output elasticity cannot be constant with respect to any input $x_i$, $i=1,\ldots,n$.
    \item[ii.] $f$ does not satisfy the PMRS property.
  \item[iii.] $f$  does not satisfy the CES property.
  \item[iv.] The production hypersurface associated with the Spillman-Mitscherlich production model $f$  has non-vanishing Gauss-Kronecker curvature.
  \item[v.] The production hypersurfaces associated with the Spillman-Mitscherlich production model $f$ is non-flat.
\end{enumerate}
\end{cor}

\begin{cor}\label{C3} Suppose $f$ is a transcendental production model defined by \eqref{A2}. Then:
\begin{enumerate}
  \item[i.] The output elasticity is constant with respect to an input $x_i$ iff $b_i=0$.
    \item[ii.] The output elasticity is constant with respect to all inputs iff $b_1=\ldots=b_n=0$.
     \item[iii.] $f$ satisfies the PMRS property iff $a_1=\ldots=a_n\neq0$ and $b_1=\ldots=b_n=0$. Moreover,
in this case, the production hypersurface associated with the transcendental production model $f$ cannot be minimal, but it has vanishing
sectional curvature if and only if $a_1=\ldots=a_n=\frac{1}{2}$.
  \item[iv.] The production hypersurface associated with the transcendental production model $f$ has vanishing Gauss-Kronecker curvature if and only if one of the following situations occurs:
\begin{enumerate}
  \item[(a)] $a_1+\ldots+a_n=1$ and $b_1=\ldots=b_n=0$;
  \item[(b)] There are two different indices $i,j\in\{1,\ldots,n\}$ such that $a_i=a_j=0$.
\end{enumerate}
  \item[v.] $f$ satisfies the CES property if and only if\\ $b_1=\ldots=b_n=0$;
  \item[vi.] The production hypersurfaces associated with the transcendental production model $f$ is flat if
and only if one of the following situations occurs:
\begin{enumerate}
  \item[(a)] $a_1=\ldots=a_n=0$;
  \item[(b)] $a_1=\ldots=a_n=\frac{1}{2}$ and $b_1=\ldots=b_n=0$.
\end{enumerate}
\end{enumerate}
\end{cor}

\section*{Acknowledgments}
This research project was supported by a grant from the "Research Center of the Female Scientific and Medical Colleges", Deanship of Scientific Research, King Saud University.

Haila ALODAN\\
Department of Mathematics,\\
King Saud University,\\
Riyadh 11495, Saudi Arabia\\
E-mail address: halodan1@ksu.edu.sa\\

Bang-Yen CHEN\\
Department of Mathematics,\\
Michigan State University,\\
East Lansing, Michigan
48824--1027, USA\\
E-mail address: chenb@msu.edu\\ \\

Sharief DESHMUKH\\
Department of Mathematics,\\
King Saud University,\\
Riyadh 11451, Saudi Arabia\\
E-mail address: shariefd@ksu.edu.sa\\

Gabriel Eduard V\^{I}LCU$^{1,2}$ \\
$^1$University of Bucharest, Faculty of Mathematics and Computer Science,\\
Research Center in Geometry, Topology and Algebra,\\
Str. Academiei 14, Sector 1,\\
Bucure\c sti 70109, Romania\\
E-mail address: gvilcu@gta.math.unibuc.ro\\
$^2$Petroleum-Gas University of Ploie\c sti,\\
Department of Mathematical Modelling, Economic Analysis and Statistics,\\
Bd. Bucure\c sti 39, Ploie\c sti 100680, Romania\\
E-mail address: gvilcu@upg-ploiesti.ro
\end{document}